\documentclass[10pt]{article}

\usepackage{hyperref}
\usepackage{amsthm,amsmath,amssymb}
\usepackage{enumerate}
\usepackage{color}

\usepackage{fullpage}
\newtheorem{theorem}{Theorem}
\newtheorem{lemma}[theorem]{Lemma}
\newtheorem{remark}[theorem]{Remark}
\newtheorem{definition}[theorem]{Definition}
\newtheorem{corollary}[theorem]{Corollary}

\newtheorem{proposition}[theorem]{Proposition}

\newcommand{\Tr}{\operatorname{Tr}}

\newcommand{\id}{\operatorname{id}}

\newcommand{\EE}{\mathbb{E}}

\newcommand{\R}{\mathbb{R}}

\newtheorem{manualtheoreminner}{Theorem}
\newenvironment{manualtheorem}[1]{%
  \renewcommand\themanualtheoreminner{#1}%
  \manualtheoreminner
}{\endmanualtheoreminner}

 \newcommand{\bary}{\mathsf{bar}}

\title{Generalized  Blaschke--Santal\'o-type inequalities, \\
without symmetry restrictions}
\author{Thomas A.~Courtade and Edric Wang \\University of California, Berkeley  }
\date{\today}

\begin{document}

\maketitle

\begin{abstract}
Nakamura and Tsuji (2024) recently investigated a many-function generalization of the functional Blaschke--Santal\'o inequality, which they refer to as a generalized Legendre duality relation.  They showed that, among the class of all even test functions, centered Gaussian functions saturate this general family of functional inequalities.  Leveraging a certain entropic duality, we give a short alternate proof of Nakamura and Tsuji's result, and, in the process, eliminate all symmetry assumptions.  As an application, we establish  a Talagrand-type inequality for the Wasserstein barycenter problem (without symmetry restrictions) originally  conjectured by Kolesnikov and Werner (\textit{Adv.~Math.}, 2022). An analogous geometric Blaschke--Santal\'o-type inequality is established for many convex bodies, again without symmetry assumptions.
\end{abstract}

\section{Introduction and Main Results}

Let $E_i \equiv \mathbb{R}^{d_i}$, $i=1, \dots, n$ be Euclidean spaces, and put $E = \oplus_{i=1}^n E_i$.  Let $\pi_i : E \to E_i$ denote the canonical coordinate projection of $E$ onto $E_i$, and for $x\in E$, define $x_i:=\pi_i(x)$. Let $Q$ be a fixed symmetric matrix on $E$, and also fix positive reals $\mathbf{c} = (c_i)_{i=1}^n$. Let $\mathcal{P}(E_i)$ denote the set of Borel probability measures on $E_i$, absolutely continuous with respect to Lebesgue measure, and having finite second moments.  If $\mu_i \in \mathcal{P}(E_i)$ is a probability measure on $E_i$ with density proportional to $e^{-f_i}\in L^1(E_i)$, then its Shannon entropy is given by
$$
h(\mu_i):= \log \left( \int_{E_i} e^{-f_i} dx_i\right)  + \int_{E_i}  f_i d\mu_i. 
$$
It is well-known that finiteness of second moments ensures the second integral exists in the Lebesgue sense; in particular,    $h(\mu_i)<\infty$.

With  $(\mathbf{c},Q)$ fixed,  we shall primarily consider   multimarginal entropy inequalities of the form  
\begin{align}
 \sum_{i=1}^n c_i h(\mu_i)   \leq  \sup_{\mu \in \pi(\mu_1, \dots, \mu_n) } \int_E \langle x, Q x\rangle d\mu  + D, \label{eq:mmEntropy}
\end{align} 
and characterize the sharp constant $D$.  Here, the supremum is over all couplings $\mu$ of $(\mu_i)_{i=1}^n$; that is, $\mu$ is a probability measure on $E$ with marginals $\pi_i \sharp \mu = \mu_i$ for $i=1,\dots,n$.   This motivates the following definitions.
\begin{definition}
Let $D(\mathbf{c},Q)$ denote the smallest constant $D$ such that \eqref{eq:mmEntropy} holds for all centered $\mu_i \in \mathcal{P}(E_i)$, $i=1,\dots, n$.   Similarly, define $D_g(\mathbf{c},Q)$ to be the smallest constant $D$ such that \eqref{eq:mmEntropy} holds for all centered Gaussian measures $(\mu_i)_{i=1}^n$.  
 \end{definition}
We remark that the centering assumption built into the above definitions comes without any loss of generality, and can be eliminated by subtracting $\langle \tau, Q \tau\rangle$ from the RHS of \eqref{eq:mmEntropy}, where $\tau = (\tau_1, \dots, \tau_n) \in E$ and $\tau_i$ is the mean of $\mu_i$.   
 
 Our main result is the following Gaussian saturation principle.
 \begin{theorem}\label{thm:GaussianSaturation}
 For any $(\mathbf{c},Q)$, it holds that $D(\mathbf{c},Q)=D_g(\mathbf{c},Q)$.
 \end{theorem}

 This result has a dual formulation as a family of functional inequalities, which we explain next.  First, we say that  a Borel function $f_i : E_i \to \mathbb{R}$, has well-defined  barycenter
$$
\bary(f_i) := \frac{1}{\int_{E_i} e^{-f_i} dx_i}\int_{E_i} x_i e^{-f_i} dx_i, 
$$
if both integrals are finite.  For   $f_i : E_i \to \mathbb{R}$, $i=1, \dots, n$, with well-defined barycenters, the concatenation   $(\bary(f_1), \dots, \bary(f_n))$ is a vector in $E$, and we may define the functional 
$$
\mathcal{Q}(f_1, \dots, f_n) := \big\langle (\bary(f_1), \dots, \bary(f_n)), Q\,  (\bary(f_1), \dots, \bary(f_n))\big\rangle.  
$$

The following is our second main result.
\begin{theorem}\label{thm:dualityAsymm}
If   Borel functions $f_i: E_i \to \mathbb{R} \cup\{+\infty\}$   are such that $\bary(f_i)$ is well-defined and 
\begin{align}
\sum_{i=1}^n c_i f_i(x_i) \geq  \langle x, Q x\rangle, ~~~~ \forall x=(x_1, \dots, x_n) \in E, \label{eq:pointwiseInequalityIntro}
\end{align}
then 
\begin{align}
 \prod_{i=1}^n \left( \int_{E_i}e^{-f_i} d x_i\right)^{c_i} \leq e^{D_g(\mathbf{c},Q)-\mathcal{Q}(f_1, \dots, f_n)}.  \label{IntegralIneqIntro}
\end{align}
This inequality is sharp.  
\end{theorem}

 Inequalities of the above form may be considered as generalized Blaschke--Santal\'o-type inequalities.  In a recent paper, Nakamura and Tsuji \cite{NT} considered this class of functional inequalities for the special case of even functions, and established the same Gaussian saturation principle for \eqref{IntegralIneqIntro} (in the case of even functions, the barycenter correction $\mathcal{Q}(f_1, \dots, f_n)$ vanishes).  Our contribution is three-fold.  First, we establish Gaussian saturation for the above family of functional inequalities; this removes the restriction to even functions imposed in \cite[Theorem 1.1]{NT}.  Second, we establish the formal duality between the above inequalities and the entropy inequalities \eqref{eq:mmEntropy}, in the sense that both classes of inequalities share the same sharp constant and Gaussian saturation property.  Finally,  our proof of Gaussian saturation proceeds by analyzing the entropy inequalities \eqref{eq:mmEntropy}.  Not only does this extend Nakamura and Tsuji's recent result, it gives a fundamentally different, and arguably simpler, proof.  
 
 \begin{remark}
The formal duality between Theorems \ref{thm:GaussianSaturation} and \ref{thm:dualityAsymm} parallels what is known for Brascamp--Lieb inequalities \cite{carlen2009subadditivity} and their forward-reverse generalizations \cite{CL21, LCCV}.  
\end{remark}

 \section{Applications}
 Motivated by a problem initially proposed by Kolesnikov and Werner \cite{KW}, we give a few brief applications.
 \subsection{Application to Wasserstein barycenters}
 Let $\gamma$ be the standard Gaussian measure on $\mathbb{R}^d$.  Let $W_2 : \mathcal{P}(\mathbb{R}^d) \times  \mathcal{P}(\mathbb{R}^d) \to \mathbb{R}_{\geq 0}$ denote the quadratic Wasserstein distance, and $D(\mu \| \gamma)$ denote the relative entropy of $\mu \in \mathcal{P}(\mathbb{R}^d)$ with respect to $\gamma$.   The first   application of Theorem \ref{thm:GaussianSaturation} is to obtain the following estimate,  which extends    Talagrand's  entropy-transport inequality \cite{Tal}.
 \begin{theorem}\label{thm:BaryTal} Let $\mu_1, \dots, \mu_n \in \mathcal{P}(\mathbb{R}^d)$ be centered, and let $\lambda_i >  0$ satisfy $\sum_{i=1}^n \lambda_i = 1$.  It holds that  
\begin{align}
  \min_{\mu \in \mathcal{P}(\mathbb{R}^d)}   \sum_{i=1}^n  \lambda_i W_2(\mu_i,\mu)^2  &\leq  2  \sum_{i=1}^n  \lambda_i (1- \lambda_i ) D(\mu_i \| \gamma). \label{eq:TalBarycenter}
\end{align}
Equality holds if and only if: (i) for $n=2$, $\mu_1 = N(0,C)$ and $\mu_2 = N(0,C^{-1})$ for positive definite $C$; (ii) for $n\geq 3$,  $\mu_i  = N(0,\id)$ for each $i=1,\dots, n$.  
\end{theorem}

The minimization problem on the left hand side of \eqref{eq:TalBarycenter} is known as the Wasserstein barycenter problem with weights $(\lambda_i)_{i=1}^n$ \cite{AC}.  A unique minimizer always exists, and is called the \emph{Wasserstein barycenter} of $(\mu_i)_{i=1}^n$ for  weights $(\lambda_i)_{i=1}^n$.  Subsequent to its introduction in \cite{AC}, the Wasserstein barycenter problem has generated    significant theoretical and practical interest; see, e.g., \cite[Ch. 9]{PC}, and references therein.  Motivated by barycenter problems and their connection to Blaschke--Santal\'o-type inequalities through duality, Kolesnikov and Werner  conjectured and established certain Talagrand-type estimates for Wasserstein barycenters in \cite{KW}.  In particular, Kolesnikov and Werner proved the special case of \eqref{eq:TalBarycenter} under the strong symmetry assumptions that the $\mu_i$'s are unconditional (i.e., each $\mu_i$ is symmetric in every coordinate), and the weights $(\lambda_i)_{i=1}^n$ are equal \cite[Theorem 7.1]{KW}.  Very recently, Nakamura and Tsuji showed that Kolesnikov and Werner's result holds under the weaker assumption that the $\mu_i$'s are symmetric \cite[Theorem 4.5]{NT} (i.e., $\mu_i(B) = \mu_i(-B)$ for Borel $B\subset \R^d$).  Hence, our Theorem \ref{thm:BaryTal}  generalizes Nakamura and Tsuji's Talagrand-type Wasserstein  barycenter inequalities, where the symmetry assumptions are weakened to simple centering, which comes without any essential loss of generality (see the remark following the statement of Theorem \ref{thm:multimarginal}, below).  In addition to the greater generality afforded, our proof  is also   simpler than those  for the symmetric cases considered in \cite{KW,NT}.

From  Agueh and Carlier's seminal work \cite{AC}, it is well-known that there is a close connection between the Wasserstein barycenter problem and a multimarginal optimal transport considered by Gangbo  and A.~\'{S}wi\c{e}ch in \cite{GC}, in the sense that for $X_i \sim \mu_i \in \mathcal{P}(\mathbb{R}^d)$, $i=1,\dots, n$, 
\begin{align}
 \min_{\mu \in \mathcal{P}(\mathbb{R}^d)}   \sum_{i=1}^n  \lambda_i W_2(\mu_i,\mu)^2  = \min_{\pi(X_1, \dots, X_n)} \EE\left[  \sum_{1 \leq i < j \leq n} \lambda_i \lambda_j |X_i - X_j|^2  \right], \label{barycenterMultimarginalEquivalence}
\end{align}
where the minimum on the right side is over couplings of random vectors $X_i \sim \mu_i$, $i=1, \dots, n$, on $(\mathbb{R}^d)^n$.   Moreover, if $\pi^{\star}$ denotes the optimal coupling on the right hand side (which is  uniquely attained), then the Wasserstein barycenter is equal to the pushforward $ T{\sharp} \pi^{\star}$, with $T(x_1, \dots, x_n) := \sum_{i=1}^n \lambda_i x_i$.   These statements can be found in \cite[Theorem 4.1 and Proposition 4.2]{AC}, with the first cited result being attributed to \cite{GC}.   See also \cite[Theorem 2.4]{KW}. In view of  this, Theorem \ref{thm:BaryTal} has an equivalent formulation as follows, which we shall prefer to work with.
\begin{manualtheorem}{4'}\label{thm:multimarginal} Let $X_i \sim \mu_i \in \mathcal{P}(\R^d)$, $i=1, \dots ,n $ be centered, and let $\lambda_i >  0$ satisfy $\sum_{i=1}^n \lambda_i = 1$.  We have  
\begin{align*}
\frac{1}{2} \inf_{\pi(X_1, \dots, X_{n})} \EE\left[  \sum_{1 \leq i < j \leq n} \lambda_i \lambda_j |X_i - X_j|^2  \right]  &\leq   \sum_{i=1}^n  \lambda_i (1- \lambda_i ) D(\mu_i \| \gamma).
\end{align*}
Equality holds if and only if: (i) for $n=2$,  $\mu_1 = N(0,C)$ and $\mu_2 = N(0,C^{-1})$ for positive definite $C$; (ii) for $n\geq 3$, $\mu_i =   N(0,\id)$ for each $i=1,\dots, n$.  
\end{manualtheorem}

To see how Theorem  \ref{thm:multimarginal}  follows from Theorem \ref{thm:GaussianSaturation}, note that we can rewrite the relative entropies in terms of Shannon entropies and second moments to reveal the equivalent inequality
$$
\sum_{i=1}^n \lambda_i (1-\lambda_i) h(\mu_i) \leq \sup_{\pi(X_1, \dots, X_{n})} \EE\left[  \frac{1}{2}\sum_{1 \leq i < j \leq n} \lambda_i \lambda_j \langle X_i , X_j\rangle  \right] + \frac{d \log(2\pi)}{2}\sum_{i=1}^n \lambda_i (1-\lambda_i). 
$$
This is of the form \eqref{eq:mmEntropy}, so the inequality of Theorem \ref{thm:multimarginal} reduces to establishing the identity 
$$
D_g(\mathbf{c},Q) =  \frac{d \log(2\pi)}{2}\sum_{i=1}^n \lambda_i (1-\lambda_i)
$$
for the appropriate choice of  $(\mathbf{c},Q)$.  This computation can be done explicitly, however we shall instead give a short direct proof of Theorem \ref{thm:multimarginal} in Section \ref{sec:BarycenterProof} which is of potential independent interest and has the benefit of characterizing the equality cases.

\subsection{Application to geometry}
With the dual equivalence between sharp constants in \eqref{eq:mmEntropy} and \eqref{IntegralIneqIntro} in hand, we record here the fact that our Talagrand-type estimate for Wasserstein barycenters is formally equivalent to the following generalization of the functional  Blaschke--Santal\'{o}  inequality. This follows by expressing the relative entropies in Theorem \ref{thm:multimarginal} in terms of Shannon entropies, and applying the general duality principle   with $c_i = \lambda_i(1-\lambda_i)$ and $E_i = \mathbb{R}^d$ for each $i=1,\dots,n$.
\begin{corollary}\label{cor:FnlBS}
Let $\lambda_i >  0$ satisfy $\sum_{i=1}^n \lambda_i = 1$, and let $f_i : \mathbb{R}^d \to \mathbb{R}$, $i=1, \dots, n$ be Borel functions with well-defined barycenters satisfying
\begin{align}
\sum_{i=1}^n \lambda_i (1-\lambda_i) f_i(x_i) \geq \sum_{1\leq i < j \leq n} \lambda_i \lambda_j \langle x_i, x_j \rangle, ~~~\forall x_1, \dots, x_n \in \mathbb{R}^d .\label{ptwiseInequality}
\end{align}
Then 
$$
e^{\mathcal{Q}(f_1, \dots, f_n)}\prod_{i=1}^n \left( \int e^{-f_i(x)}dx \right)^{\lambda_i (1-\lambda_i)} \leq (2\pi)^{\frac{d}{2} \sum_{i=1}^n \lambda_i (1-\lambda_i)}.
$$
\end{corollary}

We remark that the functional  Blaschke--Santal\'{o}  inequality  corresponds to the case $n=2$; it is  generally stated with the centering assumption $\int x e^{-f_1(x)}  dx=0$, in which case the barycenter correction term $\mathcal{Q}(f_1, f_2)$ vanishes. This was first proved in \cite{AKM}, and earlier in \cite{Ball} for the restricted setting of even functions.  Nakamura and Tsuji's main result in \cite{NT} implies the above, but only in the special case where each $f_i$ is an even function, in which case the barycenter correction term $\mathcal{Q}(f_1, \dots, f_n)$ vanishes; also, they only obtain the explicit  sharp constant in the case of equal weights.   Kolesnikov and Werner previously established  the more restrictive case where each $f_i$ is unconditional and weights are all equal \cite{KW}.   Although both \cite{NT} and \cite{KW} explore certain dual forms of their results, neither contemplate the generic duality principle that relates entropy inequalities of the form \eqref{eq:mmEntropy} to functional inequalities of the form \eqref{IntegralIneqIntro}, which generalizes that previously known for the functional Blaschke--Santal\'o inequality (see, e.g., \cite{MF}).

With the functional form of the barycenter problem now written as Corollary \ref{cor:FnlBS}, we conclude with an  application to  geometry following \cite[Section 5]{KW}, in the same way that the functional Blaschke--Santal\'o inequality may be used to deduce the classical Blaschke--Santal\'o inequality for convex bodies.   Indeed, recall that for a convex (or, more generally, star-shaped) body $C\subset \mathbb{R}^d$ containing the origin, the associated Minkowski functional is defined as
$$
p_C(x) := \inf\{ t> 0 : x \in t C\}, ~~~x\in \mathbb{R}^d.
$$
By the layer cake representation, we have the  classical identity
$$
(2 \pi)^{d/2}\frac{\operatorname{vol}_d(C) }{\operatorname{vol}_d(B_2^d)} = \int_{\mathbb{R}^d} e^{-\frac{1}{2}p_C(x)^2}dx,
$$
where $\operatorname{vol}_d$ is $d$-dimensional volume, and $B_2^d$ is the unit ball in $\mathbb{R}^d$.  Additionally, if $C$ is centered in the sense that  $\int_C x dx=0$, then another application of the layer cake representation reveals that $\bary( \frac{1}{2}p_C^2) = 0$.  A direct application of Corollary \ref{cor:FnlBS} then gives the following, which removes the  assumptions of equal weights and unconditional convex bodies made in \cite{KW}, and of equal weights and symmetric convex bodies made in \cite{NT}.

\begin{corollary}
Let $\lambda_i >  0$ satisfy $\sum_{i=1}^n \lambda_i = 1$.  If $C_1, \dots, C_n \subset \mathbb{R}^d$ are convex bodies, each containing the origin, satisfying 
$$
\frac{1}{2}\sum_{i=1}^n \lambda_i (1-\lambda_i) p_{C_i}(x_i)^2 \geq   \sum_{1\leq i < j \leq n} \lambda_i \lambda_j \langle x_i, x_j \rangle, ~~~\forall x_1, \dots, x_n \in \mathbb{R}^d
$$
and $C_1, \dots, C_{n-1}$ are centered, then 
$$
\prod_{i=1}^n \operatorname{vol}_d(C_i)^{q_i}  \leq  \operatorname{vol}_d(B_2^d) ,
$$
where $q_i := \frac{\lambda_i (1-\lambda_i) }{ \sum_{j=1}^n \lambda_j(1-\lambda_j) } $.
\end{corollary}
\begin{remark}The same holds if the bodies are star-shaped instead of convex.  Up to translation, if $n>2$, then  equality holds if only if $C_i = B_2^d$ for all $i$.  For $n=2$, equality holds precisely when $C_1$ is a centered ellipsoid, and $C_2$ is the corresponding polar body; this is the classical statement of the geometric Blaschke--Santal\'o inequality \cite{Santalo}.  
\end{remark}

 \section{Proof of Theorem \ref{thm:GaussianSaturation}}

The proof of Theorem \ref{thm:GaussianSaturation} proceeds in two steps.  First, we show a reduction to symmetric measures.  Then, we establish Gaussian saturation. In principle, the Gaussian saturation step could be accomplished by invoking Nakamura and Tsuji's \cite[Theorem 1.1]{NT} for even functions and establishing the generic duality principle.  However, as we shall see, Gaussian saturation  follows in a fairly straightforward way after taking the steps needed to establish both the duality principle and the reduction to symmetric measures.  So, we opt for this more self-contained approach.  
To this end, we introduce the following formal definition.
\begin{definition}
Let $D_s(\mathbf{c},Q)$ denote the smallest constant $D$ such that \eqref{eq:mmEntropy} holds for all symmetric $\mu_i \in \mathcal{P}(E_i)$, $i=1,\dots, n$.
\end{definition}
It is helpful to note that we always have the following ordering of $D, D_s$, and $D_g$:
$$
D(\mathbf{c},Q) \geq D_s(\mathbf{c},Q) \geq D_g(\mathbf{c},Q) > -\infty.
$$

\subsection{The equivalence $D_s(\mathbf{c},Q) = D(\mathbf{c},Q)$.}

This section contains the reduction to the symmetric measures.  We begin with several simple technical lemmas. 

\begin{lemma}\label{lem:PropertiesD} For $\delta>0$, it holds that  
$$D(\mathbf{c},Q+ \delta \id) \leq D(\mathbf{c},Q ) = \lim_{\delta \downarrow 0} D(\mathbf{c},Q+ \delta \id).$$
In particular, $\delta \mapsto D(\mathbf{c},Q+ \delta \id)$ is non-increasing and right-continuous.  The same holds when $D$ is replaced by $D_g$ or $D_s$.
\end{lemma}
 \begin{proof}
For any centered $(\mu_i)_{i=1}^n$ with finite second moments and $\delta>0$
\begin{align*}
\sum_{i=1}^n c_i h(\mu_i)  &\leq \sup_{\mu \in \pi(\mu_1, \dots, \mu_n)}  \int_E \langle x, Q   x\rangle d\mu(x) + D(\mathbf{c},Q)\\
&\leq \sup_{\mu \in \pi(\mu_1, \dots, \mu_n)}  \int_E \langle x, (Q+\delta \id) x\rangle d\mu(x)   + D(\mathbf{c},Q).
\end{align*}
So, we conclude $D(\mathbf{c},Q+\delta \id)\leq D(\mathbf{c},Q)$.  Similarly, definitions give
\begin{align*}
\sum_{i=1}^n c_i h(\mu_i)  &\leq \sup_{\mu \in \pi(\mu_1, \dots, \mu_n)}  \int_E \langle x, (Q+\delta \id)  x\rangle d\mu(x) + D(\mathbf{c},Q+\delta \id)\\
&= \sup_{\mu \in \pi(\mu_1, \dots, \mu_n)}  \int_E \langle x, Q x\rangle d\mu(x) + \delta\sum_{i=1}^n \int |x_i|^2 d\mu_i + D(\mathbf{c},Q+\delta \id).
\end{align*}
Taking $\delta\downarrow 0$, we conclude
\begin{align*}
\sum_{i=1}^n c_i h(\mu_i)  \leq  \sup_{\mu \in \pi(\mu_1, \dots, \mu_n)}  \int_E \langle x, Q x\rangle d\mu(x) + \liminf_{\delta \downarrow 0}D(\mathbf{c},Q+\delta \id). 
\end{align*}
It follows, therefore,  by  definition of $D(\mathbf{c},Q )$ that $D(\mathbf{c},Q ) \leq \liminf_{\delta \downarrow 0} D(\mathbf{c},Q+ \delta \id)$. Combination with the reverse inequality gives the claim.  The same argument works for $D_g$ and $D_s$ by restricting attention to Gaussian and symmetric measures, respectively.
\end{proof}

\begin{proposition}\label{prop:Qblocks}The following hold:
\begin{enumerate}[(i)]
\item If $\pi_i Q \pi^*_i \not\geq 0$ for some $i\in \{1,\dots, n\}$, then $D_g(\mathbf{c},Q) = +\infty$.  
\item If $\pi_i Q \pi^*_i \geq 0$ for every $i\in \{1,\dots, n\}$, then $D(\mathbf{c},Q+\delta \id_E) <\infty$ for all $\delta>0$. 
\end{enumerate}
\end{proposition}
\begin{proof}
Suppose    that $v \in E_i$ is an eigenvector of $\pi_i Q \pi_i^*$ with negative eigenvalue $\lambda<0$.  Fix a parameter $R>0$. Let $\mu_i$ be Gaussian with variance $R$ in the direction of $v$, and unit variance on the  complementary subspace.  Also, let $\mu_j = N(0,\id_{E_j})$ for each $j\in \{1,\dots, n\}\setminus\{i\}$. 
In this way, and with the help of Cauchy--Schwarz, \eqref{eq:mmEntropy} with sharp constant $D = D_g(\mathbf{c},Q)$ reads as
$$
\frac{c_i}{2}\log R \leq \lambda R + O(1 + \sqrt{R}) + D_g(\mathbf{c},Q).
$$
Since $\lambda <0$, we take $R\to \infty$ and find $D_g(\mathbf{c},Q) = +\infty$.  

Now, for the second claim, let  centered $\mu_i\in \mathcal{P}(E_i)$ have covariance   $K_i$.  Using the fact that Gaussians maximize entropy subject to covariance constraints and the assumption that $\pi_i Q \pi^*_i \geq 0$ for every $i\in \{1,\dots, n\}$, we have 
\begin{align*}
\sum_{i=1}^n c_i h(\mu_i) - \sup_{\mu \in \pi(\mu_1, \dots, \mu_n) } \int_E \langle x, (Q+ \delta\id_E) x\rangle d\mu 
&\leq \sum_{i=1}^n \frac{c_i\dim(E_i)}{2} \log\left(2 \pi e \det(K_i)^{1/\dim(E_i)}\right)  - \delta \sum_{i=1}^n \Tr(K_i)\\
&\leq \sum_{i=1}^n \frac{c_i\dim(E_i)}{2} \log\left(\frac{2 \pi e}{\dim(E_i)} \Tr(K_i) \right)  - \delta \sum_{i=1}^n \Tr(K_i).
\end{align*}
The final line is uniformly bounded from above for all $(K_i)_{i=1}^n$, so the claim follows.

\end{proof}

The following proposition contains the core argument of this subsection.    It roughly follows the same scheme as the proof of Gaussian saturation for forward-reverse Brascamp--Lieb inequalities in  \cite{LCCV}.

 \begin{proposition}\label{prop:ExtremizedImpliesGaussian} The following hold:
 \begin{enumerate}[(i)]
 \item If $D(\mathbf{c},Q) <\infty$, then $D(\mathbf{c},Q) =D_s(\mathbf{c},Q)$.    
 \item Moreover, if  there are symmetric $\mu_i \in \mathcal{P}(E_i)$, $i=1, \dots, n$ that achieve 
 \begin{align}
 \sum_{i=1}^n c_i h(\mu_i)   =  \sup_{\mu \in \pi(\mu_1, \dots, \mu_n) } \int_E \langle x, Q x\rangle d\mu  + D_s(\mathbf{c},Q), \label{eq:extremized}
 \end{align}
 then $D_s(\mathbf{c},Q) = D_g(\mathbf{c},Q)$. 
 \end{enumerate}
 \end{proposition}
 \begin{proof}
 Let $\mu_i \in \mathcal{P}(E_i)$ be centered, and let  $X_i, X'_i$ be i.i.d.~according to $\mu_i$. 
Define
$$
X^+_i = \frac{X_i + X'_i}{\sqrt{2}}; ~~X^-_i = \frac{X_i - X'_i}{\sqrt{2}}. \label{eq:Xminus}
$$
The simple but crucial ingredient is to observe that, for each $i=1,\dots, n$, 
$$
(X_i^+,X_i^-) = \left( \frac{X_i + X'_i}{\sqrt{2}},  \frac{X_i - X'_i}{\sqrt{2}} \right)  \overset{law}{=} \left( \frac{X_i + X'_i}{\sqrt{2}},  \frac{X'_i - X_i}{\sqrt{2}} \right) = (X_i^+,- X_i^-).
$$
In particular, the conditional law of $X_i^-$ given $\{X^+_i = x^+_i\}$ is symmetric  for each $x_i^+$ ($\operatorname{law}(X_i^+)$-a.s.).  Also, the law of $X^+_i$ is trivially centered.

We'll indulge in the usual abuse of notation, and denote random vectors and their distributions interchangeably (e.g., $h(X_i) = h(\mu_i)$).  Following our previous practice, we write $X\in \pi(X_1, \dots, X_n)$ to denote that $X$ is an $E$-valued random vector whose law is a coupling of $(\mu_i)_{i=1}^n$; i.e,. $\pi_i(X) \overset{law}{=}  X_i$.  Extending this, we  write $(X,X')\in \pi((X_1,X'_1), \dots, (X_n,X_n'))$ to denote that $(X,X')$ is a pair of $E$-valued random vectors, with $(\pi_i(X),\pi_i(X')) \overset{law}{=} (X_i,X_i')$.  Similar notation will be used for $X^+$ and $X^-$.  

Starting with the chain rule for entropy, we have 
\begin{align}
 \sum_{i=1}^n c_i h(X_i) &= \frac{1}{2}\sum_{i=1}^n c_i h(X^+_i) + \frac{1}{2}\sum_{i=1}^n c_i   h(X^-_i|X^+_i)  \notag \\
 &\leq \max_{(X^+,X^-)\in \pi((X^+_1,X^-_1), \dots, (X^+_n,X^-_n)) } \EE\left[  \frac{1}{2}\langle X^+, Q X^+\rangle +  \frac{1}{2}\langle X^-, Q X^- \rangle \right]
 + \frac{D(\mathbf{c},Q) + D_s(\mathbf{c},Q) }{2}\label{eq:doubling}\\
 &= \max_{(X,X')\in \pi((X_1,X'_1), \dots, (X_n,X_n')) } \EE\left[  \frac{1}{2}\langle X, Q X\rangle +  \frac{1}{2}\langle X', Q X' \rangle \right]
 + \frac{D(\mathbf{c},Q) + D_s(\mathbf{c},Q) }{2} \label{eq:rotateBack}\\
   &\leq  \max_{X \in \pi( X_1, \dots, X_n ) }  \int_E \langle X, Q X\rangle d\mu  + \frac{D(\mathbf{c},Q) + D_s(\mathbf{c},Q) }{2}. \label{eqLinExp}
\end{align}

Inequality \eqref{eq:doubling} follows, in part, due to the definition of $D(\mathbf{c},Q)$ applied to the centered   $(X_i^+)_{i=1}^n$.  The other part is explained as follows: the conditional entropy $h(X_i^-|X_i^+)$ is equal to the entropy of the conditional law of $X_i^-$ given $\{X^+_i = x_i\}$, averaged with respect to the law of $X^+_i$.  So, letting $X^-_i| x^+_i$ denote  $X^-_i$ conditioned on $\{X^+_i = x^+_i\}$, we have for any coupling $X^+ \in \pi(X^+_1 , \dots,  X^+_n )$ that 
\begin{align*}
\sum_{i=1}^n c_i   h(X^-_i|X^+_i)  \leq \EE\left[    \max_{X^- \in \pi( X^-_1|\pi_1(X^+), \dots, X^-_n|\pi_n(X^+)) }  \EE\left[   \langle X^-, Q X^-\rangle   |X^+\right] \right] + D_s(\mathbf{c},Q),
\end{align*}
where we used the definition of $D_s(\mathbf{c},Q)$, which is justified here since  the conditional law of each $X_i^-$ given $\{X^+_i = x^+_i\}$ is symmetric, $\operatorname{law}(X_i^+)$-a.s.  Invoking   measurable selection theorems, as in \cite[Cor. 5.22]{Vil2}, the optimal (conditional) coupling above   can be taken to be a regular conditional probability (see \cite[p. 42]{LCCV2}).  When this conditional coupling is hooked up with the coupling $X^+$, the resulting pair $(X^+,X^-)$ is a coupling in $\pi((X^+_1,X^-_1), \dots, (X^+_n,X^-_n))$.  The subsequent identity \eqref{eq:rotateBack} follows by defining  
$$
X = \frac{1}{\sqrt{2}}(X^+ + X^-), ~~X' = \frac{1}{\sqrt{2}}(X^+ - X^-),
$$
for which  $(X,X')  \in \pi((X_1,X'_1), \dots, (X_n,X_n')) \Leftrightarrow (X^+,X^-)\in \pi((X^+_1,X^-_1), \dots, (X^+_n,X^-_n))$, and  
$$
\langle X^+, Q X^+\rangle +   \langle X^-, Q X^- \rangle = \langle X, Q X\rangle +  \langle X', Q X' \rangle
$$
by rotation invariance of the quadratic form.   The final step \eqref{eqLinExp} is just linearity of expectation and the fact that each $X$ and $X'$ are couplings in $\pi( X_1, \dots, X_n )$.

By definition of $D(\mathbf{c},Q)$, it follows that 
$$
D(\mathbf{c},Q) \leq \frac{D(\mathbf{c},Q) + D_s(\mathbf{c},Q) }{2},
$$
and we obtain the first assertion.

To prove the second part, take $\mu_i$ to be symmetric.  Then, $X_i^+$ is symmetric also, so the display containing \eqref{eq:doubling}-\eqref{eq:rotateBack} may be modified as
\begin{align}
 \sum_{i=1}^n c_i h(X_i) &= \frac{1}{2}\sum_{i=1}^n c_i h(X^+_i) + \frac{1}{2}\sum_{i=1}^n c_i   h(X^-_i|X^+_i)  \notag \\
 &\leq \max_{(X^+,X^-)\in \pi((X^+_1,X^-_1), \dots, (X^+_n,X^-_n)) } \EE\left[  \frac{1}{2}\langle X^+, Q X^+\rangle +  \frac{1}{2}\langle X^-, Q X^- \rangle \right]
 +  D_s(\mathbf{c},Q) \label{eq:doublingSym}\\
 &= \max_{(X,X')\in \pi((X_1,X'_1), \dots, (X_n,X_n')) } \EE\left[  \frac{1}{2}\langle X, Q X\rangle +  \frac{1}{2}\langle X', Q X' \rangle \right]
 +  D_s(\mathbf{c},Q)   \notag\\
   &\leq  \max_{X \in \pi( X_1, \dots, X_n ) }  \int_E \langle X, Q X\rangle d\mu  +  D_s(\mathbf{c},Q) . \notag
\end{align}
We conclude that  if symmetric $(X_i)_{i=1}^n$ satisfy 
\begin{align}
 \sum_{i=1}^n c_i h(X_i) &=  \max_{X \in \pi( X_1, \dots, X_n ) }  \EE \langle X, Q X\rangle ]  + D_s(\mathbf{c},Q), \label{eq:Extremizers}
\end{align}
then we have equality in \eqref{eq:doublingSym}, and \eqref{eq:Extremizers} remains satisfied when $X_i$ is replace by $X^+_i$ for each $i$.  Hence,  symmetric extremizers in \eqref{eq:mmEntropy} for sharp constant $D = D_s(\mathbf{c},Q)$ are closed under rescaled convolutions.  Iterating this argument ad infinitum, the LHS of \eqref{eq:Extremizers} tends to 
$$
\sum_{i=1}^n c_i h(N(0,K_i))
$$
by the entropic central limit theorem, where $K_i$ denotes the covariance of $\mu_i$.  Moreover, since the transformation $X_i \mapsto X^+_i$ preserves covariance, the quadratic term in the RHS of \eqref{eq:Extremizers} tends to 
$$
 \max_{X \in \pi( N(0,K_1), \dots, N(0,K_n) ) }  \EE[ \langle X, Q X\rangle ]  
$$
by an application of the  central limit theorem for $W_2$.  Hence, the Gaussian measures  $\mu_i = N(0,K_i)$ also satisfy \eqref{eq:extremized}, showing $D_s(\mathbf{c},Q) \leq D_g(\mathbf{c},Q)$.  This completes the proof. 
\end{proof}

The main objective of this subsection now follows easily. 
\begin{corollary}\label{cor:DsEqualsD}
We have $D_s(\mathbf{c},Q)  = D(\mathbf{c},Q)$.
\end{corollary}
\begin{proof}
If $\pi_i Q \pi_i \not\geq 0$ for some $i$, then $D(\mathbf{c},Q)  \geq  D_s(\mathbf{c},Q) \geq D_g(\mathbf{c},Q) = +\infty$, where the last identity follows from Proposition \ref{prop:Qblocks}. So, assume $\pi_i Q \pi_i \geq 0$ for each $i$.  In this case, Proposition \ref{prop:Qblocks} ensures that $D(\mathbf{c},Q+\delta \id_E) <\infty$ for all $\delta>0$, so it follows from Lemma \ref{lem:PropertiesD} and Proposition \ref{prop:ExtremizedImpliesGaussian}(i)  that
$$
D_s(\mathbf{c},Q) = \lim_{\delta\downarrow 0} D_s(\mathbf{c},Q+\delta \id_E) = \lim_{\delta\downarrow 0} D(\mathbf{c},Q+\delta \id_E) =D(\mathbf{c},Q).
$$
\end{proof}

\subsection{The equivalence $D_s(\mathbf{c},Q) = D_g(\mathbf{c},Q)$.}

The goal of this subsection is to prove the identity $D_s(\mathbf{c},Q) = D_g(\mathbf{c},Q)$.  The proof rests primarily on the following duality principle, and several of the preparatory lemmas from the previous subsection.  
\begin{theorem}\label{thm:EntropicDuality}
Let $D \in \mathbb{R}\cup \{+\infty\}$. The following statements are equivalent.  
\begin{enumerate}[(A)]
\item If  even Borel functions $f_i: E_i \to \mathbb{R} \cup\{+\infty\}$   satisfy $e^{-f_i}\in L^1(E_i)$ and 
\begin{align}
\sum_{i=1}^n c_i f_i(x_i) \geq  \langle x, Q x\rangle, ~~~~ \forall x=(x_1, \dots, x_n) \in E, \label{eq:pointwiseInequality}
\end{align}
then 
\begin{align}
 \prod_{i=1}^n \left( \int_{E_i}e^{-f_i} d x_i\right)^{c_i} \leq e^{D}.  \label{IntegralIneq}
\end{align}
\item If $\mu_i \in \mathcal{P}(E_i)$, $i=1, \dots, n$, are symmetric then 
\begin{align}
  \sum_{i=1}^n c_i h(\mu_i) \leq \sup_{\mu \in \pi(\mu_1, \dots, \mu_n)}  \int_E \langle x, Q x\rangle d\mu(x) + D. \label{eq:entropyIneq}
\end{align}
\end{enumerate}
Moreover, log-concave $(\mu_i)_{i=1}^n$ saturate  \eqref{eq:entropyIneq} for the sharp constant $D= D_s(\mathbf{c},Q)$.
\end{theorem}

We'll need a functional counterpart to Proposition \ref{prop:Qblocks}(i).
\begin{lemma}\label{lem:PropertiesDFunc} If the sharp constant $D$ such that \eqref{IntegralIneq} holds for all admissible functions satisfying \eqref{eq:pointwiseInequality} is finite,  then $\pi_i Q \pi^*_i \geq 0$ for each $i=1,\dots, n$.  
\end{lemma}
\begin{proof}
We'll show that if $\pi_i Q \pi_i^*\not\geq 0$, then  we must have $D =+\infty$ in \eqref{IntegralIneq}.  Without loss of generality, we can consider  $i=1$.  
Suppose   that $v \in E_1$ an eigenvector of $\pi_1 Q \pi_1^*$ with negative eigenvalue $\lambda<0$ and unit length $|v|=1$.  Fix a parameter $R>0$. There is $\delta>0$ sufficiently large (not depending on $R$) such that for
$$
c_1 f_1(x_1) := \begin{cases}
\frac{\delta}{2} \langle x_1, (\id_{E_1} - v\otimes v ) x_1 \rangle + \frac{\lambda}{2} | \langle  v ,  x_1\rangle  |^2 & \mbox{if $ |  \langle  v ,  x_1\rangle | \leq R$}\\
+ \infty & \mbox{if $ | \langle  v ,  x_1\rangle  | > R$}
\end{cases},  ~~x_1\in E_1
$$
and 
$$
c_i f_i(x_i) := \frac{\delta}{2} | x_i|^2 ,  ~~x_i\in E_i, ~~i=2,\dots, n,
$$
we have 
$$
\sum_{i=1}^n c_i f_i(x_i) \geq \langle x, Q x\rangle, ~~\forall x=(x_1, \dots, x_n)\in E. 
$$
 Direct computation shows 
$$
\int_{E_i} e^{-f_i}dx_i = \begin{cases}
\left( \frac{c_i 2 \pi}{\delta}\right)^{(\dim(E_i)-1)/2}  \int_{-R}^R e^{-\frac{\lambda}{2 c_i} t^2} dt & \mbox{if $i=1$}\\
\left( \frac{c_i 2 \pi}{\delta}\right)^{\dim(E_i)/2} & \mbox{if $i=2,\dots, n$.}
\end{cases}
$$
Since $\lambda < 0$ and $\delta$ does not depend on $R$, we find that we must have $D =+\infty$ in \eqref{IntegralIneq} by letting $R\to \infty$.  
\end{proof}

\begin{proof}[Proof of Theorem \ref{thm:EntropicDuality}]
$(B) \Rightarrow (A)$. Let $(f_i)_{i=1}^n$ be even Borel functions satisfying \eqref{eq:pointwiseInequality}.  Take  (symmetric) $\mu_i$ to have density proportional to $e^{-{f_i}} $. Assume for now that each $\mu_i$ has compact support (this is always possible by setting $f_i = +\infty$ outside the support of $\mu_i$, which preserves \eqref{eq:pointwiseInequality}).  Due to the assumption of compact support, $\mu_i$ has finite second moments, and therefore well-defined entropy $h(\mu_i) < \infty$.    In particular, 
$$
h(\mu_i) =   \log \left( \int e^{-f_i} d x_i \right)  + \int  f_i d\mu_i.
$$
By \eqref{eq:pointwiseInequality}, one may see that each $f_i$ is bounded from below on its (compact) essential support.  Since $h(\mu_i)<\infty$,   it follows that $f_i  \in L^1(\mu_i)$. 
 Thus, using definitions and (B), we have 
\begin{align*}
\sum_{i=1}^n c_i \left( \log \left( \int e^{-f_i} d x_i \right)  + \int  f_i d\mu_i  \right)&=  \sum_{i=1}^n c_i h(\mu_i) \\
&\leq  \max_{\mu \in \pi(\mu_1, \dots, \mu_n)} \int \langle x, Q x\rangle d\mu  +D\\
&\leq   \sum_{i=1}^nc_i \int f_i  d\mu_i  +D.
\end{align*}
Cancelling the $\int f_i  d\mu_i$ terms on both sides, we conclude (A) holds.  The assumption that $\mu_i$ is compactly supported can be removed using monotone convergence.

We now show $(A) \Rightarrow (B)$.    We can assume the sharp constant $D$ in \eqref{IntegralIneq} is finite, else the claim is trivially true.  Therefore, by Lemma \ref{lem:PropertiesDFunc}, we may assume that $\pi_i Q \pi_i^* \geq 0$ for each $i=1,\dots, n$.  

Fix symmetric measures $\mu_i \in \mathcal{P}(E_i)$, and assume without loss of generality that $h(\mu_i) > -\infty$.   
By the multimarginal Kantorovich duality, for any $\epsilon>0$, there exist even $f_i\in L^1(\mu_i)$ satisfying \eqref{eq:pointwiseInequality} and 
\begin{align}
\sum_{i=1}^n c_i \int f_i(x_i) d\mu_i  <   \max_{\mu \in \pi(\mu_1, \dots, \mu_n)} \int \langle x, Q x\rangle d\mu + \epsilon. \label{maxCouplingOrig}
\end{align}
Note that \eqref{eq:pointwiseInequality} implies
$$
c_i f_i(x_i) \geq \sup_{x_j \in E_j, j\neq i} \left\{ \langle x, Q x \rangle - \sum_{j=1, i\neq j}^n c_i f_i(x_i) \right\}.
$$
Since $\pi_i Q \pi_i^* \geq 0$, the RHS is a pointwise supremum of even functions that are convex in $x_i$, and is therefore  even and convex in $x_i$.  Both \eqref{eq:pointwiseInequality} and  inequality \eqref{maxCouplingOrig} are preserved by replacing $c_i f_i$ with this even convex function.  Repeating this same argument for each $i=1, \dots, n$ shows that we may assume without loss of generality that each $f_i$ is even and convex.  In fact, we may add $\epsilon' |x_i|^2$ to each $f_i$; this obviously preserves \eqref{eq:pointwiseInequality} for any $\epsilon'>0$; and, for $\epsilon'>0$ sufficiently small, \eqref{maxCouplingOrig} will also be preserved since the second moments $(\mu_i)_{i=1}^n$ are finite.  Therefore, we may assume without loss of generality that each $f_i$ is uniformly convex.  In particular, each $e^{-f_i} \in L^1(E_1)$.

Now, starting with the variational representation for entropy and   (A), we have 
\begin{align*}
\sum_{i=1}^n c_i h(\mu_i)  &\leq \sum_{i=1}^n c_i \left( \int f_i d\mu_i +  \log\left( \int_{E_i} e^{-f_i} dx_i \right)   \right) 
\leq   \max_{\mu \in \pi(\mu_1, \dots, \mu_n)} \int \langle x, Q x\rangle d\mu + \epsilon + D.
\end{align*}
We conclude (A)$\Rightarrow$(B).

If $\pi_iQ \pi^*_i \not\geq 0$ for some $i$, then Proposition \ref{prop:Qblocks} ensures that $D_s(\mathbf{c},Q)= +\infty$, and it suffices to consider  Gaussian (and therefore log-concave) measures to achieve  saturation.  On the other hand, if $\pi_iQ \pi^*_i  \geq 0$ for all $i=1, \dots, n$, then the argument above shows that it suffices to consider even convex $(f_i)_{i=1}^n$ satisfying \eqref{eq:pointwiseInequality} to saturate \eqref{IntegralIneq} for sharp constant $D$.  Since sharp constants in (A) and (B) are now established to be the same, returning to the argument in the proof of (B)$\Rightarrow$(A) shows that it suffices to consider log-concave $(\mu_i)_{i=1}^n$ to saturate \eqref{eq:entropyIneq}.
\end{proof}

\begin{lemma}
Let $(\nu_k)_{k\geq 1}$ be a sequence of zero-mean log-concave probability measures on $\mathbb{R}^d$, with uniformly bounded second moments $\int |x|^2 d\nu_k \leq M$ for all $k\geq 1$.  There exists a probability measure $\nu^*$ such that $W_2(\nu_{k_n}, \nu^*)\to 0$ for some suitable subsequence $(\nu_{k_n})_{n\geq 1}$. 
\end{lemma}
\begin{proof}
By the uniform bound on second moments, the collection $(\nu_k)_{k\geq 1}$ is tight, so by Prokhorov's theorem there exists a probability measure $\nu^*$ and a subsequence $(\nu_{k_n})_{n\geq 1}$ such that $\nu_{k_n} \to \nu^*$ weakly.  To show convergence in $W_2$, it suffices to show that the second moments of $(\nu_{k_n})_{n\geq 1}$ converge.  For the latter, it suffices to establish uniform integrability.  Toward this end, recall that higher-order central moments of log-concave measures are controlled by lower moments \cite{Bor1975}.  That is, there is a constant $C$ (depending only on dimension) such that 
$$
\int |x|^4 d\nu_{k_n} \leq C \left( \int |x|^2 d\nu_{k_n}  \right)^2, ~~n\geq 1.
$$
Thus, 
$$
\int_{|x|\geq r} |x|^2 d\nu_{k_n} \leq \frac{1}{r^2} \int_{|x|\geq r} |x|^4 d\nu_{k_n} \leq \frac{C}{r^2}  M^2~~~\Rightarrow~~~\lim_{r\to \infty} \sup_{n\geq 1} \int_{|x|\geq r} |x|^2 d\nu_{k_n} =0.
$$
This is the desired uniform integrability, and it follows that $W_2(\nu_{k_n}, \nu^*)\to 0$.
\end{proof}

We now turn our attention to the main objective of this subsection.
\begin{corollary}\label{cor:DsEqualsDg}
$D_s(\mathbf{c},Q) = D_g(\mathbf{c},Q)$.
\end{corollary}
\begin{proof}
By Proposition \ref{prop:Qblocks}(i), it suffices to assume $\pi_i Q \pi_i \geq 0$ for every $i\in \{1,\dots, n\}$.   So, with this assumption made, fix $\delta>0$ and let $\mu_i \in \mathcal{P}(E_i)$ be symmetric with   covariance   $K_i$, $i=1,\dots, n$.  Recall from the calculation in the proof of Proposition \ref{prop:Qblocks} that, with  the assumption that $\pi_i Q \pi_i \geq 0$ for every $i\in \{1,\dots, n\}$, we have 
\begin{align*}
\sum_{i=1}^n c_i h(\mu_i) - \sup_{\mu \in \pi(\mu_1, \dots, \mu_n) } \int_E \langle x, (Q+ \delta\id_E) x\rangle d\mu 
&\leq \sum_{i=1}^n \frac{c_i\dim(E_i)}{2} \log\left(\frac{2 \pi e}{\dim(E_i)} \Tr(K_i) \right)  - \delta \sum_{i=1}^n \Tr(K_i).
\end{align*}
Hence, in computing $D_s(\mathbf{c},Q+ \delta \id_E)$, it suffices to consider measures $\mu_i$ with $\Tr(K_i)\leq C(\mathbf{c},\delta)$, where   $C(\mathbf{c},\delta)$ is a constant depending only on $\delta$ and the coefficients $\mathbf{c}$.  Indeed, if any one $\Tr(K_i)$ tends to infinity, then the upper bound  will eventually be less than $D_g(\mathbf{c},Q+ \delta \id_E)$.  In particular, to compute $D_s(\mathbf{c},Q+ \delta \id_E)$, it suffices to restrict attention to symmetric measures $\mu_i \in \mathcal{P}(E_i)$ with second moments uniformly bounded by $C(\mathbf{c},\delta)$.   Moreover, by Theorem \ref{thm:EntropicDuality}, it  suffices to further restrict attention to log-concave measures.  

So, let $(\mu_i^{(k)})_{i=1}^n$, $k\geq 1$ be a sequence of symmetric log-concave measures, with second moments uniformly bounded by $C(\mathbf{c},\delta)$, such that 
$$
\lim_{k\to \infty} \left( \sum_{i=1}^n c_i h(\mu^{(k)}_i) - \sup_{\mu \in \pi(\mu^{(k)}_1, \dots, \mu^{(k)}_n) } \int_E \langle x, (Q+ \delta\id_E) x\rangle d\mu\right)  =  D_s(\mathbf{c},Q+ \delta \id_E).  
$$
By passing to a subsequence if necessary, there exists  $\mu_i^*$ such that $W_2(\mu_i^{(k)}, \mu_i^*)\to 0$ for each $i=1,\dots, n$.  In particular,  
$$
\lim_{k\to\infty} \sup_{\mu \in \pi(\mu^{(k)}_1, \dots, \mu^{(k)}_n)} \int_E \langle x, (Q+ \delta\id_E) x\rangle d\mu = 
 \sup_{\mu \in \pi(\mu^{*}_1, \dots, \mu^{*}_n)} \int_E \langle x, (Q+ \delta\id_E) x\rangle d\mu .
$$
By lower semicontinuity of relative entropy and convergence of second moments, we have  the upper semicontinuity estimate for Shannon entropy
$$
\limsup_{k\to \infty} \sum_{i=1}^n c_i h(\mu^{(k)}_i)  \leq  \sum_{i=1}^n c_i h(\mu^{*}_i) .
$$
Thus, we conclude 
$$
 \sum_{i=1}^n c_i h(\mu^{*}_i)   =  \sup_{\mu \in \pi(\mu^{*}_1, \dots, \mu^{*}_n) } \int_E \langle x, (Q+ \delta\id_E) x\rangle d\mu  + D_s(\mathbf{c},Q+ \delta \id_E).
$$
The measures $(\mu_i^*)_{i=1}^n$ are obviously symmetric, so an application of Proposition \ref{prop:ExtremizedImpliesGaussian}(ii) ensures $D_s(\mathbf{c},Q+\delta \id_E) = D_g(\mathbf{c},Q+\delta \id_E)$.  Using Lemma \ref{lem:PropertiesD}, we may take $\delta\downarrow 0$ to conclude the proof. 
\end{proof}

Theorem \ref{thm:GaussianSaturation}   now follows immediately  by combining  Corollaries \ref{cor:DsEqualsD} and \ref{cor:DsEqualsDg}.

\section{Proof of Theorem \ref{thm:dualityAsymm}}

\begin{proof}[Proof of Theorem \ref{thm:dualityAsymm}]
Begin by noting that if $\mu_i$ has mean $\tau_i \in E_i$, $i=1, \dots, n$, then by translation-invariance of Shannon entropy, \eqref{eq:mmEntropy} with sharp constant $D = D_g(\mathbf{c},Q)$ is equivalent to 
\begin{align}
  \sum_{i=1}^n c_i h(\mu_i) \leq \sup_{\mu \in \pi(\mu_1, \dots, \mu_n)}  \int_E \langle x, Q x\rangle d\mu(x)  - \langle \tau, Q \tau\rangle + D_g(\mathbf{c},Q) \label{eq:entropyIneqMeanShift}
\end{align}
for $\tau := (\tau_1, \dots, \tau_n) \in E$.  Indeed the incorporation of the correction term $\langle \tau, Q \tau\rangle$ makes both sides invariant to translations of the measures $(\mu_i)_{i=1}^n$.  

The proof now follows that of Theorem \ref{thm:EntropicDuality}, which we repeat for completeness.  In particular,   let $(f_i)_{i=1}^n$ be Borel functions with well-defined barycenters satisfying \eqref{eq:pointwiseInequalityIntro}.  Take  $\mu_i$ to have density proportional to $e^{-{f_i}} $. Assume for now that each $\mu_i$ has compact support (this is always possible by setting $f_i = +\infty$ outside the support of $\mu_i$, which preserves \eqref{eq:pointwiseInequalityIntro}).  Evidently, $\mu_i$ has mean $\bary(f_i)$ and, due to the assumption of compact support, $\mu_i$ has finite second moments, and therefore well-defined entropy $h(\mu_i) < \infty$.    In particular, 
$$
h(\mu_i) =   \log \left( \int e^{-f_i} d x_i \right)  + \int  f_i d\mu_i.
$$
By \eqref{eq:pointwiseInequalityIntro}, one may see that each $f_i$ is bounded from below on its essential support.  Since $h(\mu_i)<\infty$,   it follows that $f_i  \in L^1(\mu_i)$. 
 Thus, using definitions and \eqref{eq:entropyIneqMeanShift}, we have 
\begin{align*}
\sum_{i=1}^n c_i \left( \log \left( \int e^{-f_i} d x_i \right)  + \int  f_i d\mu_i  \right)&=  \sum_{i=1}^n c_i h(\mu_i) \\
&\leq  \max_{\mu \in \pi(\mu_1, \dots, \mu_n)} \int \langle x, Q x\rangle d\mu - \mathcal{Q}(f_1,\dots, f_n) +D_g(\mathbf{c},Q)\\
&\leq   \sum_{i=1}^nc_i \int f_i  d\mu_i - \mathcal{Q}(f_1,\dots, f_n) +D_g(\mathbf{c},Q).
\end{align*}
Cancelling the $\int f_i  d\mu_i$ terms on both sides, we conclude \eqref{IntegralIneqIntro} holds.  The assumption that $\mu_i$ is compactly supported can be removed using monotone convergence.  Now, the constant is seen to be sharp by restricting attention to even functions (which have barycenter zero) and applying Theorem \ref{thm:EntropicDuality}, which has sharp constant $D = D_g(\mathbf{c},Q)$ by Theorem \ref{thm:GaussianSaturation}.
\end{proof}

\section{Proof of Theorem \ref{thm:multimarginal}}\label{sec:BarycenterProof}
We give a direct, independent proof of Theorem \ref{thm:multimarginal}.  It rests on two well-known results.  The first is the symmetrized Talagrand inequality due to Fathi \cite{MF}; see the recent paper \cite{CFM} for a short martingale proof. 
\begin{lemma}[Symmetrized Talagrand inequality]\label{lem:SymmTal}
Let $\mu, \nu \in \mathcal{P}(\R^d)$ be centered.  We have  
\begin{align}
\frac{1}{2} W_2(\mu, \nu)^2 \leq     D(\mu \| \gamma) + D(\nu \| \gamma) . \label{eq:symmTalagrand}
\end{align}
Moreover,  equality holds if and only if $\mu =  N(0,C)$ and $\nu =  N(0,C^{-1})$ for positive definite $C$.  
\end{lemma}
We remark that the symmetrized Talagrand inequality coincides with  Theorem   \ref{thm:multimarginal} in the special case where $n=2$.  Thus, Theorem  \ref{thm:multimarginal} may be regarded as its natural  multimarginal generalization.

The second  lemma we shall need is the known displacement convexity property enjoyed by Wasserstein barycenters  \cite[Proposition 7.7]{AC}.  The statement below follows from  combining \cite[Eq. (6.4)]{KW}  with \eqref{barycenterMultimarginalEquivalence}.  
\begin{lemma}[Displacement convexity of Wasserstein barycenters]
Let $X_i \sim \mu_i \in \mathcal{P}(\R^d)$, $i=1, \dots ,n $ be centered, and let $\lambda_i >  0$ satisfy $\sum_{i=1}^n \lambda_i = 1$.  If $\bar{\mu}$ denotes the  Wasserstein barycenter of $(\mu_i)_{i=1}^n$ for weights $(\lambda_i)_{i=1}^n$, then 
\begin{align}
\sum_{i=1}^n \lambda_i D(\mu_i\| \gamma )  &\geq  
  D(\bar{\mu} \| \gamma)  + \frac{1}{2} \inf_{\pi(X_1, \dots, X_{n})} \EE\left[  \sum_{1 \leq i < j \leq n} \lambda_i \lambda_j |X_i - X_j|^2  \right] .  \label{eq:dispConvexityGen}
\end{align}
\end{lemma}

\noindent We are now ready to prove Theorem \ref{thm:multimarginal}.
\begin{proof}[Proof of Theorem \ref{thm:multimarginal}]
The proof is by induction on $n$.  The case $n=1$ is degenerate, and the case $n=2$ is precisely the symmetrized Talagrand inequality.  So, assume the statement is true for some $n\geq 2$, and we'll prove it for $n+1$.  For convenience, define $\bar{p} := 1-p$ for $p\in [0,1]$.

Put $\lambda'_i = \lambda_i / \bar{\lambda}_{n+1}$ for $i=1,\dots,n$,  and let $\bar{\mu}$ denote the Wasserstein barycenter of $(\mu_i)_{i=1}^n$ for weights $(\lambda'_i)_{i=1}^n$.  Observe that, for $X_i \sim \mu_i$, $i=1,\dots, n+1$, we have 
\begin{align}
&\frac{1}{2} \inf_{\pi(X_1, \dots, X_{n+1})} \EE\left[  \sum_{1 \leq i < j \leq n+1} \lambda_i \lambda_j  |X_i - X_j|^2  \right] \notag\\
&\leq  \frac{\bar{\lambda}_{n+1} \lambda_{n+1} }{2}W_2(\mu_{n+1},\bar{\mu})^2 + \frac{\bar{\lambda}_{n+1}}{2} \inf_{\pi(X_1, \dots, X_{n})} \EE\left[  \sum_{1 \leq i < j \leq n}  \lambda'_i \lambda'_j |X_i - X_j|^2  \right]  \label{ineq1} \\
&\leq \bar{\lambda}_{n+1} \lambda_{n+1}  D(\mu_{n+1} \| \gamma) + \bar{\lambda}_{n+1} \lambda_{n+1}  D(\bar{\mu}  \| \gamma) +  \frac{\bar{\lambda}_{n+1}}{2} \inf_{\pi(X_1, \dots, X_{n})} \EE\left[  \sum_{1 \leq i < j \leq n} \lambda'_i \lambda_j '|X_i - X_j|^2  \right]  \label{ineq2} \\
&\leq \bar{\lambda}_{n+1} \lambda_{n+1}   D(\mu_{n+1} \| \gamma)  +  \lambda_{n+1}   \sum_{i=1}^n \lambda_i D (\mu_i \|\gamma)  +  \frac{\bar{\lambda}_{n+1}^2  }{2} \inf_{\pi(X_1, \dots, X_{n})} \EE\left[  \sum_{1 \leq i < j \leq n} \lambda'_i \lambda_j '  |X_i - X_j|^2  \right]  \label{ineq3}\\
&\leq \bar{\lambda}_{n+1} \lambda_{n+1}   D(\mu_{n+1} \| \gamma)  +  \lambda_{n+1}   \sum_{i=1}^n \lambda_i D(\mu_i \|\gamma)  +   \bar{\lambda}_{n+1}^2   \sum_{i=1}^n \lambda'_i \bar{\lambda'_i}D(\mu_i\|\gamma)  \label{ineq4}\\
&= \bar{\lambda}_{n+1} \lambda_{n+1}   D(\mu_{n+1} \| \gamma) +    \sum_{i=1}^n \bar{\lambda}_i \lambda_i D (\mu_i \|\gamma)  .  \notag
\end{align}
Inequality \eqref{ineq1} is the elementary identity for quadratic forms
$$
  \sum_{1 \leq i < j \leq n+1}  \lambda_i \lambda_j |x_i - x_j|^2  = 
\bar{\lambda}_{n+1} \lambda_{n+1}  |x_{n+1} - \bar{x}|^2 + \bar{\lambda}_{n+1}  \sum_{1 \leq i < j \leq n}  \lambda'_i \lambda'_j  |x_i - x_j|^2, ~~\mbox{where $\bar{x} := \sum_{i=1}^n \lambda'_i x_i$},
$$
combined with the fact that the RHS corresponds to restricting the set of couplings of $X_1,\dots,X_{n+1}$ on the LHS to those where $\EE\left[  \sum_{1 \leq i < j \leq n} \lambda'_i \lambda'_j |X_i - X_j|^2  \right]$ is minimized (here, we recall the connection between the barycenter $\bar{\mu}$ and the optimal coupling for this multimarginal transport problem).  Inequality \eqref{ineq2} is the symmetrized Talagrand inequality \eqref{eq:symmTalagrand} applied to $W_2(\mu_{n+1},\bar{\mu} )^2$.  Inequality \eqref{ineq3} is an application of \eqref{eq:dispConvexityGen} with weights $(\lambda_i')_{i=1}^n$, and \eqref{ineq4} is the inductive hypothesis. 

For $n=2$, the equality cases are those of the symmetrized Talagrand inequality, and are known.  For $n=3$, the fourth inequality above is the symmetrized Talagrand inequality applied to measures $\mu_1,\mu_2$.  So, for equality, we must have $\mu_1 = N(0,C)$ and $\mu_2= N(0,C^{-1})$ by Lemma \ref{lem:SymmTal}.  However, by symmetry of the argument, we must have the same inverse relationship between the covariances of any pair $\mu_i,\mu_j$.  Hence, the only possibility is that $C = \id$.  For $n\geq 4$, we proceed by induction.  Again, the fourth inequality in the above argument shows that in the case of $(n+1)$ marginals, equality implies any (and therefore every) subset of $n$ distinct marginals must be standard normal by the assumed equality case for $n$, so the claim follows for $n+1$.
\end{proof}

\subsection*{Acknowledgments}
T.C. thanks D.~Cordero--Erausquin and A.~Eskenazis for stimulating conversations at the Hausdorff Research Institute for Mathematics, during the Fall 2024 program on Information theory, Boolean functions, and lattice problems where some of these results were presented.

\subsection*{Funding} Parts of this work were supported by NSF CCF-1750430.

\end{document}